\documentclass{amsart}
\usepackage{graphicx} 
\usepackage{yufei}
\usepackage{comment}

\DeclareMathOperator{\dist}{dist}
\DeclareMathOperator{\bzero}{{\bf 0}}
\DeclareMathOperator{\Span}{span}
\DeclareMathOperator{\diam}{diam}

\title{Finer control on relative sizes of iterated sumsets}
\author{Jacob Fox}
\author{Noah Kravitz}
\author{Shengtong Zhang}
\address{Fox, Zhang: Stanford University, Stanford, CA 94305, USA}
\email{\{jacobfox,stzh1555\}@stanford.edu}

\address{Kravitz: Princeton University, Princeton, NJ 08540, USA}
\email{nkravitz@princeton.edu}

\begin{document}
\maketitle
\begin{abstract}
Inspired by recent questions of Nathanson, we show that for any infinite abelian group $G$ and any integers $m_1, \ldots, m_H$, there exist finite subsets $A,B \subseteq G$ such that $|hA|-|hB|=m_h$ for each $1 \leq h \leq H$.  We also raise, and begin to address, questions about the smallest possible cardinalities and diameters of such sets $A,B$.
\end{abstract}

\section{Introduction}

\subsection{The main problem}

For a subset $A$ of an abelian group $G$, denote the $h$-fold iterated sumset of $A$ by
$$hA:=\{a_1+\cdots+a_h: a_1, \ldots, a_h \in A\}.$$
For finite $A$, consider the sequence
$$|A|, |2A|, |3A|, \ldots$$
of the sizes of the iterated sumsets of $A$. 
Foundational results in additive combinatorics, such as the Pl\"unnecke--Ruzsa inequality, concern the possible growth rates of such a sequence.  Nathanson~\cite{nathanson} recently raised a new set of questions about the possible \emph{relative} growth rates of such sequences for different choices of $A$.

A classical theorem of Khovanskii~\cite{khov1,khov2} (see also~\cite{nathanson2,NR}) says that for each $A$, the quantity $|hA|$ is a polynomial function of $h$ once $h$ is sufficiently large.  It follows that for any finite subsets $A,B \subseteq G$, either $|hA|<|hB|$ for all sufficiently large $h$, or $|hA|=|hB|$ for all sufficiently large $h$, or $|hA|>|hB|$ for all sufficiently large $h$.  Nathanson, who was concerned primarily with the integer setting $G=\mathbb{Z}$, asked what kinds of oscillations in the signs of $|hA|-|hB|$ are possible if we restrict our attention to $1 \leq h \leq H$ for some natural number $H$.  The second author~\cite{noah} provided constructions showing that in any fixed infinite abelian group $G$, all sign patterns are possible.

\begin{theorem}[\cite{noah}]\label{thm:noah}
Let $G$ be an infinite abelian group, and let $H \in \mathbb{N}$ and $\epsilon_1, \ldots, \epsilon_H \in \{-,+\}$. Then there exist finite sets $A,B \subseteq G$ such that
$$|hA|-|hB| \quad \text{has the same sign as $\epsilon_h$}$$
for each $1 \leq h \leq H$.
\end{theorem}

The paper~\cite{noah} also established a few strengthenings of this theorem: One can obtain the analogous statement with $d$ sets instead of just the two sets $A,B$, and in the integer setting one can also prescribe the equality $|hA|=|hB|$ for some values of $h$.  See \cite{anne} for related work.

The purpose of the present paper is to refine Theorem~\ref{thm:noah} and to consider some related questions.

\subsection{More precise constructions}

Our main result shows that we can prescribe not only the signs of the quantities $|hA|-|hB|$ but also their exact values.  For integers $a \leq b$, we write $[a,b]:=\{a,a+1,a+2, \ldots, b\}$.

\begin{theorem}\label{thm:construction-integers}
Let $H \in \mathbb{N}$ and $m_1, \ldots, m_H \in \mathbb{Z}$.  Then there exist sets $A,B \subseteq \left[0,60H^2 \sum_h |m_h|\right]$, each of size at most $2+60H \sum_h |m_h|$, such that
$$|hA|-|hB|=m_h$$
for each $1 \leq h \leq H$.
\end{theorem}

\begin{theorem}\label{thm:construction-positive-char}
Let $p$ be a prime, and let $H \in \mathbb{N}$ and $m_1, \ldots, m_H \in \mathbb{Z}$.  For all sufficiently large $N$, there exist sets $A,B \subseteq (\mathbb{Z}/p\mathbb{Z})^N$ such that
$$|hA|-|hB|=m_h$$
for each $1 \leq h \leq H$.
\end{theorem}
The proof of Theorem \ref{thm:construction-positive-char} shows that we may take $N=H+\left\lceil 10\log_p\left(1+H^{H^3}\sum_h |m_h| \right) \right\rceil$.

Combining these two theorems, as in~\cite[Section 2.1]{noah}, yields the following corollary for infinite abelian groups; we also obtain the same conclusion for finite abelian groups that are sufficiently large relative to $m_1, \ldots, m_H$.

\begin{corollary}\label{cor:construction-general}
Let $G$ be an infinite abelian group, and let $H \in \mathbb{N}$ and $m_1, \ldots, m_H \in \mathbb{Z}$. Then there exist finite sets $A,B \subseteq G$ such that
$$|hA|-|hB|=m_h$$
for each $1 \leq h \leq H$.
\end{corollary}

With just a bit more notation, the previous three results generalize easily to the case of more than two sets.  For example, the many-set version of Corollary~\ref{cor:construction-general} is as follows; again, a sufficiently large abelian group works just as well as an infinite abelian group.

\begin{corollary}\label{cor:many-sets}
Let $G$ be an infinite abelian group, and let $d,H \in \mathbb{N}$ and $\vec m_1, \ldots, \vec m_H \in \mathbb{Z}^d$.  Then there exist finite sets $A_1, \ldots, A_d \subseteq G$ such that
$$(|hA_1|, \ldots, |hA_d|)-\vec m_h$$
is a constant vector for each $1 \leq h \leq H$.
\end{corollary}

Since the arguments in the two-set and many-set versions are identical, we present proofs only for the former.

\subsection{Questions about efficiency}
Let us return to the initial motivating problem, namely, finding finite sets $A,B \subseteq \mathbb{Z}$ such that the quantities $|hA|-|hB|$ have prescribed signs for all $1 \leq h \leq H$.  We know that for any choice of signs, it is possible to find such sets $A,B$.  It is natural to ask how ``efficient'' such examples can be, in the sense of how small we can make $A,B$.  Two reasonable ways of measuring smallness are by the cardinalities of $A,B$ and by their diameters.

With this in mind, let $\kappa(H)$ denote the smallest natural number $k$ such that for each choice of signs $\epsilon_1, \ldots, \epsilon_H \in \{-,+\}$, there are sets $A,B \subseteq \mathbb{Z}$ each of size at most $k$ such that $|hA|-|hB|$ has the same sign as $\epsilon_h$ for each $1 \leq h \leq H$.  Likewise, let $\nu(H)$ denote the smallest natural number $N$ such that for each choice of signs $\epsilon_1, \ldots, \epsilon_H \in \{-,+\}$, there are sets $A,B \subseteq [0,N]$ such that $|hA|-|hB|$ has the same sign as $\epsilon_h$ for each $1 \leq h \leq H$.  Determining the asymptotics of the functions $\kappa(H), \nu(H)$ seems interesting and difficult.

The sets constructed in the proof of Theorem~\ref{thm:noah} are quite large and deliver poor upper bounds on $\kappa(H), \nu(H)$.  The construction behind Theorem~\ref{thm:construction-integers} delivers more reasonable upper bounds.  We are also able to establish some lower bounds using ideas from convex geometry and an effective Khovanskii theorem of Granville and Walker~\cite{GW}.  The following two theorems summarize the current state of affairs.


\begin{theorem}\label{thm:kappa-bounds}
For $H \geq 3$, we have $\sqrt{H/\log H} \ll \kappa(H) \ll H$.
\end{theorem}

\begin{theorem}\label{thm:nu-bounds}
For $H \geq 3$, we have $H+1 \leq \nu(H) \ll H^3$.
\end{theorem}

Our proof of the lower bound in Theorem~\ref{thm:nu-bounds} actually shows a slightly stronger statement: If $H \geq 3$ and $A,B \subseteq \mathbb{Z}$ are finite sets satisfying $|(H-2)A|<|(H-2)B|$, $|(H-1)A|>|(H-1)B|$, and $|HA|>|HB|$, then at least one of $A,B$ has diameter at least $H+1$.

For simplicity of notation, we omit floor and ceiling signs where there is no risk of confusion.

\section{Main constructions}

In this section we prove Theorems~\ref{thm:construction-integers} and~\ref{thm:construction-positive-char}.  In general the quantities $|hA|,|hB|$ grow somewhat erratically for small values of $h$.  One key idea is that when the sets $A,B$ have a special form, we can express the differences $|hA|-|hB|$ as linear combinations of some auxiliary parameters and then solve the resulting linear system.  The argument is particularly clean in the integer setting and a bit more involved in the positive-characteristic setting.

\subsection{The integers}

We begin with the integer setting.

\begin{proof}[Proof of Theorem~\ref{thm:construction-integers}]
We set $$N:=60H^2 \sum_{h = 1}^H |m_h|$$
(the choice of the constant $60$ is not important)
and define the interval $$I:=\left[\left(1-\frac{1}{2H}\right)N+1,N\right].$$  Our sets will have the form
$$A:=\{0,1\} \cup (I \setminus A'), \quad B:=\{0,1\} \cup (I \setminus B'),$$
where $A',B'$ are carefully-chosen subsets of the middle third of $I$.  For all such choices of $A',B'$ we have
$$j(I \setminus A')=j(I \setminus B')=jI$$
for all $j \geq 2$, which gives the identity
$$hA=\bigcup_{j=0}^h ((h-j)\{0,1\}+j(I \setminus A'))=[0,h] \cup ([0,h-1]+(I \setminus A')) \cup \bigcup_{j=2}^h ([0,h-j]+jI),$$
and likewise for $hB$.  For $1 \leq h \leq H$, the shortness of the interval $I$ guarantees that the first union over $j$ is in fact a disjoint union, so we can write
$$|hA|-|hB|=|[0,h-1]+(I \setminus A')|-|[0,h-1]+(I \setminus B')|.$$
Notice that we have expressed $|hA|-|hB|$ in terms of $2$-fold sumsets, rather than higher-order sumsets.

We take $A'$ (respectively, $B'$) to consist of $\alpha_r$ (respectively, $\beta_r$) intervals of length $r$, for $1 \leq r \leq H$, with the property that the left endpoints of these intervals are all spaced out at least $2H$ apart from one another.  Set $\gamma_r:=\beta_r-\alpha_r$.  Then for each $1 \leq h \leq H$, we have
$$|[0,h-1]+(I \setminus A')|=|I|+h-1-\sum_{r= h}^{H} \alpha_r (r-h+1),$$
and likewise for $|[0,h-1]+(I \setminus B')|$, which gives
\begin{equation}\label{eq:linear-expression}
|hA|-|hB|=\sum_{r=h}^{H} (\beta_r-\alpha_r) (r-h+1)=\sum_{r=h}^{H} \gamma_r (r-h+1).
\end{equation}
We would like to choose the $\gamma_r$'s to make this quantity equal $m_h$ for each $1 \leq h \leq H$.  This linear system in the $\gamma_r$'s is upper-triangular with all $1$'s on the main diagonal, so it has a solution in the integers.  Explicitly, we find that the solution is
\begin{equation*}\label{eq:formula-for-m's}
\gamma_H=m_H, \quad \gamma_{H-1}=m_{H-1}-2m_H, \quad \text{and} \quad \text{$\gamma_r=m_{r}-2m_{r+1}+m_{r+2}$
for $1 \leq r \leq H-2$}.
\end{equation*}
For each $1 \leq r \leq H$, take $\{\alpha_r,\beta_r\}=\{0,|\gamma_r|\}$ such that $\beta_r-\alpha_r=\gamma_r$.  It follows that the total number of intervals in $A',B'$ is
\begin{equation}\label{eq:sum-of-gamma's-bound}
\sum_r (\alpha_r+\beta_r)=\sum_r |\gamma_r| \leq 4\sum_h |m_h|.
\end{equation}
Recalling our choice of $N$, the middle third of $I$ has length $\frac{|I|}{3}=\frac{N}{6H}=10H\sum_h |m_h|$, so we can space out the desired intervals comprising $A',B'$.
\end{proof}

\subsection{Positive characteristic}

In the previous subsection, the shortness of the interval $I$ let us express $hA, hB$ as disjoint unions of pieces that were easier to understand individually.  Such disjointness is not possible in positive characteristic, and we will have to work harder to build a substitute structure that ``spreads out'' as we take higher-order sumsets.

Let us set up some notation before proceeding to the proof.  For $x \in \mathbb{Z}/p\mathbb{Z}$, let $|x|$ denote the smallest nonnegative integer that is congruent to $x$ or $-x$ modulo $p$.  The $\ell^1$-distance on $(\mathbb{Z}/p\mathbb{Z})^M$ is defined by $$\dist((x_1, \ldots, x_M),(y_1, \ldots, y_M)):=|x_1-y_1|+\cdots+|x_M-y_M|,$$
and we write $$\mathcal{B}_R(x):=\{y \in (\mathbb{Z}/p\mathbb{Z})^M: \dist(x,y) \leq R\}$$ for the $\ell^1$-ball of radius $R$ around a point $x \in (\mathbb{Z}/p\mathbb{Z})^M$.  Notice that
$$|\mathcal{B}_R(x)| \leq (2M+1)^R \leq (3M)^R$$
for all $M \geq 1$.  We write $\bzero:=(0, \ldots, 0)\in (\mathbb{Z}/p\mathbb{Z})^M$.

\begin{proof}[Proof of Theorem~\ref{thm:construction-positive-char}]
We set
$$M:=N-H=\left \lceil 10\log_p\left(1+H^{H^3}\sum_h |m_h| \right) \right\rceil$$
and write
$$(\mathbb{Z}/p\mathbb{Z})^{N}=V \times W,$$
where $V:=(\mathbb{Z}/p\mathbb{Z})^H$ and $W:=(\mathbb{Z}/p\mathbb{Z})^M$; one should think of $(\mathbb{Z}/p\mathbb{Z})^{N}$ as consisting of copies of $W$ indexed by the elements of $V$.  Fix a basis $e_1, \ldots, e_H$ of $V$.  We will construct the desired subsets of $(\mathbb{Z}/p\mathbb{Z})^{N}$ as
$$A:=\left(\bigcup_{i=1}^{H-1} (\{e_i\} \times \mathcal{B}_i({\bf 0})) \right) \cup (\{e_H\} \times W \setminus A'), \quad B:=\left(\bigcup_{i=1}^{H-1} (\{e_i\} \times \mathcal{B}_i({\bf 0})) \right) \cup (\{e_H\} \times W \setminus B'),$$
where $A',B'$ are carefully chosen small subsets of $W$.  In particular, we will have $|A'|,|B'|<|W|/2$, so that the Pigeonhole Principle gives
$$j(W \setminus A')=j(W \setminus B')=W$$
for all $j \geq 2$.  As in the previous subsection, we can expand
$$hA=\bigcup_{i_1 \leq \cdots \leq i_h} (\{e_{i_1}+\cdots+e_{i_h}\} \times (X_{i_1}+\cdots+X_{i_h})),$$
where $X_i:=\mathcal{B}_i(\bzero)$ if $i<H$ and $X_H:=W \setminus A'$.

This union is very much not disjoint (since $pe_i=\bzero$), so we must pay attention to how often the various indices $i_j$ are repeated.  If $H$ appears as an index multiple times, then the $W$-component is all of $W$ since $2X_H=W$.  Likewise, whenever there are at least $p$ equal indices $i_j$, then replacing $p$ such indices by $H$ gives another set of indices $i'_j$ such that $e_{i_1}+\cdots+e_{i_h}=e_{i'_1}+\cdots+e_{i'_h}$ and at least $p$ of the $i'_j$'s are equal to $H$; then, by the previous observation, the contribution of $(i_1, \ldots, i_h)$ to $hA$ is contained in the contribution of $(i'_1, \ldots, i'_h)$, and the former can be safely ignored.  Accounting for these reductions and conditioning on how many indices are equal to $H$, we can write $A$ as the disjoint union of the following three sets:\footnote{This decomposition plays the same role as the distinction $j=0$, $j=1$, $j \geq 2$ played in the proof in the previous subsection.}
\begin{enumerate}
    \item (\emph{no indices equal to $H$}) the disjoint union of the sets $$\{e_{i_1}+\cdots+e_{i_{h}}\} \times B_{i_1+\cdots+i_h}({\bf 0}),$$ where $i_1 \leq \cdots \leq i_h<H$ and no index is repeated $p$ or more times;
    \item (\emph{one index equal to $H$}) the disjoint union of the sets $$\{e_{i_1}+\cdots+e_{i_{h-1}}+e_H\} \times (B_{i_1+\cdots+i_{h-1}}({\bf 0})+(W \setminus A')),$$
    where $i_1 \leq \cdots \leq i_{h-1}<H$ and no index is repeated $p$ or more times;
    \item (\emph{at least two indices equal to $H$}) the union (not necessarily disjoint) of the sets
    $$\{e_{i_1}+\cdots+e_{i_{h-2}}+2e_H\} \times W,$$
    where $i_1 \leq \cdots \leq i_{h-2}\leq H$.
\end{enumerate}
The disjointness of these contributions is a consequence of the distinctness of their $V$-coordinates.  When we express $hB$ analogously, the first and third contributions are the same as the first and third contributions for $hA$.  So we find that
\begin{equation}\label{eq:diff-in-pos-char}
|hA|-|hB|=\sum^*_{i_1, \ldots, i_{h-1}} |\mathcal B_{i_1+\cdots+i_{h-1}}({\bf 0})+(W \setminus A')|-|\mathcal B_{i_1+\cdots+i_{h-1}}({\bf 0})+(W \setminus B')|,
\end{equation}
where $\sum^*$ runs over tuples of indices $i_1 \leq \cdots \leq i_{h-1}<H$ such that no index is repeated $p$ or more times.

We are now ready to specify the sets $A',B' \subseteq W$.  Set $s_r:=\lceil r/(p-1) \rceil$, so that $s_1, \ldots, s_{H-1}$ is a prefix of the infinite sequence
$$\underbrace{1,\ldots, 1}_{p-1}, \underbrace{2,\ldots, 2}_{p-1}, \ldots,$$
and define the partial sums $t_r:=\sum_{i=1}^{r-1} s_i$, which are strictly increasing with $r$.  We take $A'$ (respectively, $B'$)  to consist of $\alpha_r$ (respectively, $\beta_r$) copies of $\mathcal{B}_{t_r}$, for $1 \leq r \leq H$, with the property that the center-points of these balls are all spaced out at least $3t_H$ apart from one another (with respect to the $\ell^1$-distance).  Set $\gamma_r:=\beta_r-\alpha_r$.  Notice that the complement of $\mathcal{B}_R(\bzero)+(W \setminus A')$ consists of $\alpha_r$ copies of $\mathcal{B}_{t_r-R}$ for each $r$ with $t_r \geq R$, and the same goes for $\mathcal{B}_R(\bzero)+(W \setminus B')$.

Consider the summands in \eqref{eq:diff-in-pos-char}.  The term for $(i_1, \ldots, i_{h-1})=(s_1, \ldots, s_{h-1})$ contributes exactly
$$|\mathcal B_{t_h}({\bf 0})+(W \setminus A')|-|\mathcal B_{t_h}({\bf 0})+(W \setminus B')|=\sum_{r=h}^H \gamma_r |\mathcal{B}_{t_r-t_h}|;$$
recall that $|\mathcal{B}_{0}|=0$.  For each other sequence of indices $(1_1, \ldots, i_{h-1})$, we have $i_1+\cdots+i_{h-1}>t_h$, and the contribution to \eqref{eq:diff-in-pos-char} is
$$\sum_{r: t_r \geq i_1+\cdots+i_{h-1}} \gamma_r |\mathcal{B}_{t_r-(i_1+\cdots+i_{h-1})}|;$$
notice that $\gamma_1, \ldots, \gamma_h$ do not appear and that the coefficients of the $\gamma_r$'s that do appear are all at most $|\mathcal{B}_{t_H}| \leq (3M)^{t_H}$.  There are (crudely) at most $H^{H-1}$ tuples of indices in the latter category.  Putting everything together, we can express $|hA|-|hB|$ as $\gamma_h$ plus an integer linear combination of $\gamma_{h+1}, \ldots, \gamma_H$, where the coefficients of $\gamma_{h+1}, \ldots, \gamma_H$ are all of size at most $$Q:=H^{H-1}(3M)^{t_H}.$$
If we set $|hA|-|hB|=m_h$ for each $1 \leq h \leq H$, then, as before, we obtain an upper-triangular integer system of linear equations in the $\gamma_r$'s, with all $1$'s on the main diagonal and all off-diagonal entries having size at most $Q$.  Thus this linear system has a solution in the integers, and by starting at $\gamma_H$ and working our way down to $\gamma_1$ we can see inductively that
$$|\gamma_r| \leq (QH)^{H-r} \sum_{h} |m_h|$$
for each $1 \leq r \leq H$. 
As before, for each $1 \leq r \leq H$, take $\{\alpha_r,\beta_r\}=\{0,|\gamma_r|\}$ such that $\beta_r-\alpha_r=\gamma_r$.  It follows that the total number of balls in $A',B'$ is
$$\sum_r (\alpha_r+\beta_r)=\sum_r |\gamma_r| \leq 2(QH)^H \sum_h |m_h| =2H^{H^2} (3M)^{Ht_H} \sum_h |m_h|.$$

It remains to check that $(\mathbb{Z}/p\mathbb{Z})^M$ contains a set of at least this many points, all spaced out at least $3t_H$ apart, to serve as the center-points of the balls of $A',B'$.  Take any maximal subset $S \subseteq (\mathbb{Z}/p\mathbb{Z})^M$ of points that are all spaced out at least $3t_H$ apart.  Then the balls of radius $3t_H$ centered at the points of $S$ together cover all of $(\mathbb{Z}/p\mathbb{Z})^M$, and in particular
$$|S| \geq p^M \cdot |\mathcal{B}_{3t_H}|^{-1} \geq p^M (3M)^{-3t_H}.$$
Using $t_H \leq H^2$, we conclude by noting that the choice of $M$ guarantees
$$p^M \geq 2H^{H^2} (3M)^{(H+3)t_H} \sum_h |m_h|,$$
with room to spare.
\end{proof}

\section{The necessity of using large sets}
In this section we prove Theorems~\ref{thm:kappa-bounds} and~\ref{thm:nu-bounds} about bounds on $\kappa(H),\nu(H)$.  For the lower bound on $\kappa(H)$, we require the following lemma based on convex geometry.

\begin{lemma}\label{lem:lattice}
Let $G$ be any abelian group, and let $k,H \in \mathbb{N}$.  Then there are at most $O(H)^{k^2}$ different sequences
$$|A|, |2A|, \ldots, |HA|$$
as $A$ ranges over the $k$-element subsets of $G$.
\end{lemma}

\begin{proof}
We proceed by an encoding argument.  Let $A=\{a_1, \ldots, a_k\} \subseteq G$.  Then the sequence $|A|, |2A|, \ldots, |HA|$ is determined by the set
$$X(A):=\{(x_1, \ldots, x_k) \in [-H,H]^k: x_1a_1+\cdots+x_ka_k=0\} \subseteq \mathbb{Z}^k.$$
Now define the lattice
$$\Lambda(A):=\Span_{\mathbb{Z}}(X(A)).$$
Since $X(A)=\Lambda(A) \cap [-H,H]^k$, the sequence $|A|, |2A|, \ldots, |HA|$ is determined by $\Lambda(A)$.  By construction, $\Lambda(A)$ has a basis contained in $[-H,H]^k$.  The lemma now follows from the observation that there are at most
$$\sum_{\ell=0}^k (2H+1)^{k \ell} \leq O(H)^{k^2}$$
linearly independent subsets of $[-H,H]^k$ and hence at most $O(H)^{k^2}$ possible bases for $\Lambda(A)$.
\end{proof}

We mention the following alternative proof of Lemma~\ref{lem:lattice}, suggested by Noga Alon. It is well-known among experts (see, e.g., the discussion in~\cite{nathanson3}) that every set $A$ of $k$ integers is $2H$-Freiman-isomorphic to a set of $k$ integers contained in $[0,O(H)^k]$; since the number of $k$-element subsets of $[0,O(H)^k]$ is $O(H)^{k^2}$, there are at most $O(H)^{k^2}$ different sequences $|A|, |2A|, \ldots, |HA|$.

\begin{proof}[Proof of Theorem~\ref{thm:kappa-bounds}]
We start with the lower bound on $\kappa(H)$.  Lemma~\ref{lem:lattice} tells us that there are at most $O(H)^{2k^2}$ different sequences
$$|A|-|B|, |2A|-|2B|, \ldots, |HA|-|HB|$$
as $A,B$ range over subsets of $\mathbb{Z}$ each of size at most $k$.  In order for such sequences to cover all of the $2^H$ sign patterns $\{-,+\}^H$, we must have $2k^2$ larger than a suitable constant times $H/\log H$.  Thus $\kappa(H) \gg \sqrt{H/\log H}$, as desired.

For the upper bound on $\kappa(H)$, we will modify the construction from Theorem~\ref{thm:construction-integers} with $(m_1, \ldots, m_H)$ ranging over $\{-1,1\}^H$.  The two modifications are that we will place the intervals of $A',B'$ right next to each other (instead of allowing extra space between them) and that we will ``thin out'' the leftmost and rightmost thirds of the interval $I$.  The details are as follows.

Fix a choice of $(m_1, \ldots, m_H) \in \{-1,1\}^H$, and take $\alpha_r, \beta_r, \gamma_r$ (for $1 \leq r \leq H$) as in the proof of Theorem~\ref{thm:construction-integers}.  Let $s_1, \ldots, s_{\sum_r \alpha_r}$ be the sequence consisting of $\alpha_1$ copies of $1$, followed by $\alpha_2$ copies of $2$, and so on, up to $\alpha_H$ copies of $H$; define $t_1, \ldots, t_{\sum_r \beta_r}$ analogously.  Now consider the set
$$A'':=\left\{ \sum_{i=0}^j (s_i+1): 0 \leq j \leq \sum_r \alpha_r \right\},$$
and obtain $B''$ from the $t_i$'s analogously.  Notice that $A''$ (respectively, $B''$) has exactly $\alpha_r$ (respectively, $\beta_r$) gaps of length $r$ for each $1 \leq r \leq H$ and no other gaps.  We claim that $A'',B''$ are small and have similar diameters.  First, we have
$$|A''|+|B''|=2+\sum_r (\alpha_r+\beta_r)=2+\sum_r |\gamma_r| \leq 2+4\sum_r|m_r| \leq 2+4H$$
by \eqref{eq:sum-of-gamma's-bound}.
Next, the diameters of $A'',B''$ are
$$\diam(A'')=\sum_r (r+1)\alpha_r \quad \text{and} \quad \diam(B'')=\sum_r (r+1)\beta_r.$$
The sum is
$$\diam(A'')+\diam(B'')=\sum_r(r+1)(\alpha_r+\beta_r)=\sum_r (r+1)|\gamma_r| \leq 4(H+1)\sum_h |m_h| \ll H^2$$
by another application of \eqref{eq:sum-of-gamma's-bound}.  The difference is
$$\diam(B'')-\diam(A'')=\sum_r(r+1)(\beta_r-\alpha_r)=\sum_r (r+1)\gamma_r=2\sum_r r \gamma_r -\sum_r (r-1) \gamma_r=2m_1-m_2$$
by the formula \eqref{eq:linear-expression}, and this difference has absolute value at most $3$ since $|m_1|=|m_2|=1$.  Now set $$D:=\max\{\diam(A''), \diam(B'')\},$$
and define
$$\widetilde{A}:=A'' \cup [\max(A''), D] \quad \text{and} \quad \widetilde{B}:=B'' \cup [\max(B''), D].$$
It follows from the above observations that $D \ll H^2$, that $|\widetilde{A}|, |\widetilde{B}| \ll H$, that $\widetilde{A}, \widetilde{B}$ each have minimum value $0$ and maximum value $D$, and that $\widetilde{A}$ (respectively, $\widetilde{B}$) has exactly $\alpha_r$ (respectively, $\beta_r$) gaps of length $r$ for each $1 \leq r \leq H$ and no other gaps.

The next observation is that there is a set $X_D \subseteq [0,D]$ of size $|X_D| \ll \sqrt{D} \ll H$ such that $X_D+X_D=[0,D]+[0,D]=[0,2D]$.  One example of such a set $X_D$ is
$$[0,D'] \cup \{iD': 0 \leq i \leq D/D'\} \cup [D-D',D] $$
with $D':=\left\lfloor \sqrt{D} \right \rfloor$.  We are now ready to describe the final sets in our construction.  Define
$$A:=\{0,1\} \cup (10HD+X_D) \cup ((10H+1)D+\widetilde{A}) \cup ((10H+2)D+X_D),$$
and likewise for $B$.  The argument from the proof of Theorem~\ref{thm:construction-integers} ensures that $|hA|-|hB|=m_h$ for each $1 \leq h \leq H$, and we see from above that $|A|,|B| \ll H$.  Since we have such a construction for all choices of $(m_1, \ldots, m_H) \in \{-1,1\}^H$, we conclude that $\kappa(H) \ll H$.
\end{proof}

We now turn to $\nu(H)$.  We require the following effective Khovanskii theorem of Granville and Walker~\cite{GW}.

\begin{theorem}[\cite{GW}]\label{thm:effective-khovanskii}
Let $N \geq 3$ be a natural number.  Then for every subset $A \subseteq [0,N]$, there are integers $a=a(A),b=b(A)$ such that $|hA|=ah+b$ for all $h \geq N-2$.
\end{theorem}

\begin{proof}[Proof of Theorem~\ref{thm:nu-bounds}]
Theorem~\ref{thm:construction-integers} with $(m_1, \ldots, m_H)$ ranging over $\{-1,1\}^H$ immediately implies the upper bound $\nu(H) \ll H^3$.  For the lower bound on $\nu(H)$, Theorem~\ref{thm:effective-khovanskii} tells us that if $A,B$ are subsets of $[0,N]$, then $|hA|-|hB|$ is an affine-linear function of $h$ for all $h \geq N-2$.  In particular, if $N \leq H$, then it is not possible to have $|(H-2)A|<|(H-2)B|$, $|(H-1)A|>|(H-1)B|$, and $|HA|>|HB|$; it follows that $\nu(H) \geq H+1$.
\end{proof}

\section{Further questions}

We conclude with a few open questions for future research.  The most pressing problem, of course, is determining the asymptotic growth rates of the quantities $\kappa(H),\nu(H)$.  Our upper bound on $\nu(H)$ seems particularly ripe for improvement.

In our lower-bound proof for $\nu(H)$, we exhibited specific sign patterns $\epsilon_1,\ldots, \epsilon_H$ that cannot be achieved  by sets $A,B$ of small diameter.  In the setting of $\kappa(H)$, it would likewise be interesting to determine specific sign patterns that cannot be achieved by sets $A,B$ of small cardinality.

One could also study analogues of $\kappa(H),\nu(H)$ in the positive-characteristic setting; for the analogue of $\nu(H)$, one would want to use the dimension of a set as a ``substitute'' for the diameter.  Our construction for Theorem~\ref{thm:construction-positive-char} provides some upper bounds, and one could obtain weak lower bounds by counting arguments.

\section*{Acknowledgements}
This material is based upon work supported by the National Science
Foundation under Grant No. DMS-1928930, while the first and third authors were in
residence at the Simons Laufer Mathematical Sciences Institute in
Berkeley, California, during the Spring 2025 semester.  The first author was supported by NSF Award DMS-2154129. The second author was supported in part by the NSF Graduate Research Fellowship Program under
grant DGE–203965.

We thank Noga Alon for helpful conversations.

\end{document}